\documentclass{amsart}
\usepackage{dsfont}
\usepackage{graphicx}
\newtheorem{theorem}{Theorem}[section]
\newtheorem{lemma}[theorem]{Lemma}

\theoremstyle{definition}
\newtheorem{definition}[theorem]{Definition}

\newtheorem{corollary}{Corollary}

\theoremstyle{remark}
\newtheorem{remark}[theorem]{Remark}

\numberwithin{equation}{section}

\begin{document}

\title{Topological pressure of proper map}

\author{Dongkui Ma*}\thanks{* Corresponding author}
\author{Nuanni Fan}


\address{Department of Mathematics, South China University of Technology,
Guangzhou 510641, P.R. China}
\email{dkma@scut.edu.cn, fannuanni@126.com}





\begin{abstract}
Based on the Carath\'{e}odory -Pesin structure theory\cite{Pesin}, we introduce three notions of topological pressure of a proper map and provide some properties of these notions. For the proper map of a locally compact separable metric space, we prove some variational principles and give some applications. These are the extensions of results of Pesin, Takens and Verbiski, etc.
\end{abstract}

\maketitle

\section{ Introduction }
Let $f$ be a continuous map acting on a compact metric space $X$ and $\varphi$ a continuous function on $X.$ The notion of topological pressure of $\varphi$ was brought to the theory of dynamical systems by Ruelle\cite{Ruelle} and Walters\cite{Walters1}, and it was further developed by Pesin and Pitskel\cite{PP}. The topological pressure is a key notion in dynamical systems and dimension theory. In\cite{Pesin}, Pesin used the dimension approach to the notion of topological pressure, which was based on the Carath\'{e}odory structure\cite{Ca}(we call it the Carath\'{e}odory -Pesin structure, or briefly, C-P structure). It is a very powerful tool to study dimension theory and dynamical systems. For a proper map, Patr\~{a}o\cite{Patr}, Ma and Cai\cite{MC} introduced some notions of topological entropy.

In this paper, by using the C-P structure, we introduce three notions of topological pressure for a proper map of a metric space. They are the extensions of the classical topological pressures introduced by Walters\cite{Walters1}, Pesin and Pitskel\cite{PP} respectively. Some properties of these notions are provided. For the proper map of a locally compact separable metric space, we prove some variational principles and give some applications. These extend results of Pesin\cite{Pesin}, Takens and Verbitski\cite{TV}, etc.

This paper is organized as follows. In section 2, we introduce the notions of the topological pressure, the lower and upper capacity topological pressure and give some basic properties of them. In section 3, we give some further properties. In section 4, we give some variational principles and applications.

\section{Topological pressure and lower and upper capacity topological pressure introduced in this paper and their some basic properties }

 In this section, by using the C-P structure\cite{Pesin}, the topological pressure and lower and upper capacity topological pressure are introduced for the proper map of a metric space.

 Let $X$ be a topological space and $f:X\rightarrow X$ a proper map, i.e., $f$ is a continuous map such that the pre-image by $f$ of any compact set is compact. An open set is called an admissible open set if the closure or the complement of it is compact. An admissible cover of $X$ is an open and finite cover $\mathcal{U}$ of $X$ such that, for each $U\in\mathcal{U}$, $U$ is an admissible open set.

Let $(X,d)$ be a metric space and denote by $B(x,\delta)$ the open ball centered at $x$ with radius $\delta>0$. The metric $d$ is called admissible\cite{Patr} if the following conditions are satisfied:

(1) If $\mathcal{U}_{\delta}=\{B(x_{1},\delta),\cdots,B(x_{k},\delta)\}$ is a cover of $X$, for every $\delta\in (a,b)$, where $0<a<b$, then there exists $\delta_{\varepsilon}\in (a,b)$ such that $\mathcal{U}_{\delta_{\varepsilon}}$ is admissible.

(2) Every admissible cover of $X$ has a Lebesgue number.

In \cite{MC}, the authors proved that
for a metric space ~$(X,d),$ every admissible cover of $X$ has a ~Lebesgue number, so the (2) in the definition of admissible metric can be deleted. From \cite{Patr}, we see that if $d$ is an admissible metric, then for any $\varepsilon>0$ there exists an admissible cover such that the diameter of this cover is less than $\varepsilon.$
It is easy to see that, if $(X,d)$ is compact, then $d$ is automatically admissible.

 Let $(X,d)$ be a metric space and $f:X\rightarrow X$ a proper map. Given an admissible cover $\mathcal{U}$ of $X$,
 denote by $\mathcal{S}_{m}(\mathcal{U})$ the set of all strings $\mathbf{U}=
(U_{i_{0}}, U_{i_{1}},\cdots, U_{i_{m-1}})$ of length $m=m(\mathbf{U}),$ where $U_{i_j}\in \mathcal{U}, J=0,1,\cdots, m-1.$ We put $\mathcal{S}=\mathcal{S}(\mathcal{U})=\mathop
\bigcup \limits_{m\geq0} \mathcal{S}_{m}(\mathcal{U})$.

To a given string $\mathbf{U}= ( U_{i_{0}}, U_{i_{1}}, \cdots, U_{i_{m-1}} ) \in \mathcal{S}(\mathcal{U})$ we associate the set
$$ X(\mathbf{U})=\{x\in X:f^{j}(x)\in U_{i_{j}},j=0,1,\cdots,m(U)-1\}.$$
It is easy to see that $X(\mathbf{U})=\mathop{\bigcap\limits^{m(U)-1}_{j=0}}f^{-j}U_{i_{j}}$, then $X(\mathbf{U})$ is an admissible open set. Let $\varphi\in C(X,\mathbb{R})$ be bounded, where $C(X,\mathbb{R})$ denotes the space of real-valued continuous functions of $X$. Denote
$(S_{n}\varphi)(x)=\sum_{k=0}^{n-1}\varphi(f^{k}(x)).$
Define the collection of subsets
$$\mathcal{F}=\mathcal{F}(\mathcal{U})=\{X(\mathbf{U}):\mathbf{U}\in \mathcal{S}(\mathcal{U})\}$$
and three functions $\xi,\eta,\psi:\mathcal{S}(\mathcal{U}) \rightarrow
\mathbb{R}^{+}$ as follows
\begin{eqnarray*}
\xi(\mathbf{U})&=&\exp\left(\sup_{x\in
X(\mathbf{U})}(S_{m(\mathbf{U})}\varphi)(x)\right),\\
\eta(\mathbf{U})&=&\exp(-m(\mathbf{U})),\\
\psi(\mathbf{U})&=&m(\mathbf{U})^{-1}.
\end{eqnarray*}
It is easy to verify that the sets $\mathcal{S},\mathcal{F}$ and the functions $\xi,\eta,$ and $\psi$ determine a C-P structure $\tau=\tau(\mathcal{U})=(S,\mathcal{F},\xi,\eta,\psi)$ on $X$. We say that a collection of strings $\mathcal{G}$ covers a set $Z\subset X$ if $\mathop{\bigcup }\limits_{\mathbf{U}\in \mathcal{G}}X(\mathbf{U})\supset Z$. For any set $Z\subset X$ and $\alpha \in \mathbb{R}$, define
\begin{eqnarray*}
M(Z,\alpha,\varphi,\mathcal{U},N)&:=&\inf_{\mathcal{G}}\left\{\sum_{\mathbf{U}\in
\mathcal{G}}\xi(\mathbf{U})\eta(\mathbf{U})^{\alpha}\right\}\\
&=&\inf_{\mathcal{G}}\left \{\sum_{\mathbf{U}\in \mathcal{G}}\exp\left (-\alpha
m(\mathbf{U})
+ \sup_{x\in X(\mathbf{U})}(S_{m(\mathbf{U})}\varphi)(x)\right)\right\},
\end{eqnarray*}
and the infimum is taken over all finite or countable collections of strings $\mathcal{G}\subset S(\mathcal{U})$ such that $m(\mathbf{U})\geq N$ for all $\mathbf{U} \in \mathcal{G}$ and $\mathcal{G}$ covers $Z$. Let
$$m(Z,\alpha,\varphi,\mathcal{U})=\mathop{\lim}\limits_{N\rightarrow +\infty} M(Z,\alpha,\mathcal{U},N).
$$
For every real numbers $\alpha$ introduce
$$\underline{r}(Z,\alpha,\varphi,\mathcal{U})=\mathop{\underline{\lim}}_{N\rightarrow
\infty}R(Z,\alpha,\varphi,\mathcal{U},N),$$
$$\overline{r}(Z,\alpha,\varphi,\mathcal{U})=\mathop{\overline{\lim}}_{N\rightarrow
\infty}R(Z,\alpha,\varphi,\mathcal{U},N),$$
where
$$R(Z,\alpha,\varphi,\mathcal{U},N)=\inf_{\mathcal{G}}\left\{\sum_{\mathbf{U}\in \mathcal{G}}\exp\left(-\alpha
N+\sup_{x\in X(\mathbf{U})}(S_{N}\varphi)(x)\right)\right\},$$
the infimum is taken over all collections of strings $\mathcal{G}\subset S(\mathcal{U})$ such that $m(\mathbf{U})= N$ for all $\mathbf{U}\in \mathcal{G}$ and $\mathcal{G}$ covers $Z$. By the definition of C-P structure, define
\begin{eqnarray*}
P_{Z}(\varphi,\mathcal{U})&:=&\inf\{\alpha:m(Z,\alpha,\varphi,\mathcal{U})=0\}=\sup\{\alpha:m(Z,\alpha,\varphi,\mathcal{U})=\infty\},\\
\underline{CP}_{Z}(\varphi,\mathcal{U})&:=&\inf\{\alpha:\underline{r}(Z,\alpha,\varphi,\mathcal{U})=0\}=\sup\{\alpha:\underline{r}(Z,\alpha,\varphi,\mathcal{U})=\infty\},\\
\overline{CP}_{Z}(\varphi,\mathcal{U})&:=&\inf\{\alpha:\overline{r}(Z,\alpha,\varphi,\mathcal{U})=0\}=\sup\{\alpha:\overline{r}(Z,\alpha,\varphi,\mathcal{U})=\infty\}.
\end{eqnarray*}

\begin{lemma}(\cite{MC})\label{2}
Let ~$(X,d)$ be a metric space, then every admissible cover of $X$ has a ~Lebesgue number.
\end{lemma}

\begin{theorem}\label{1}~~
Let~$(X,d)$ be a metric space and $d$ an admissible metric, $f:X\rightarrow X$ a proper map, $\varphi\in C(X,\mathbb{R})$ bounded. Then for any ~$Z\subset X$, the following limites exist:
\[P_{Z}(\varphi):=\lim_{|\mathcal{U}|\rightarrow0}P_{Z}(\varphi, \mathcal{U}),\]
\[\underline{CP}_{Z}(\varphi):=\lim_{|\mathcal{U}|\rightarrow0}\underline{CP}_{Z}(\varphi,\mathcal{U}),\]
\[\overline{CP}_{Z}(\varphi):=\lim_{|\mathcal{U}|\rightarrow0}\overline{CP}_{Z}(\varphi,\mathcal{U}),\]
where $\mathcal{U}$ is admissible cover and $|\mathcal{U}|$ denote the diameter of $\mathcal{U},$
i.e., ~$|\mathcal{U}|
=\max\{diam(\mathbf{U}):\mathbf{U}\in\mathcal{U} \}.$
\end{theorem}

\begin{proof}
We use the similar method as that of \cite{Pesin}. By the Lemma \ref{2}, let $\mathcal{V}$ be an admissible cover of $X$ with diameter smaller than the Lebesgue number of admissible cover $\mathcal{U}.$ One can see that each element $V\in \mathcal{V}$ is contained in some element $U(V)\in \mathcal{U}.$ To any string $\mathbf{V}=(V_{i_{0}},\cdots,V_{i_{m}})\in \mathcal{S}(\mathcal{V})$ we associate the string $\mathbf{U}(\mathbf{V})=(U(V_{i_0}),\cdots,U(V_{i_m}))\in \mathcal{S}(\mathcal{U}).$ If $\mathcal{G}\subset \mathcal{S}(\mathcal{V})$ covers a set $Z\subset X$ then $\mathbf{U}(\mathcal{G})=\{\mathbf{U}(\mathbf{V}):\mathbf{V}\in \mathcal{G}\}\subset \mathcal{S}(\mathcal{U})$ also covers $Z.$ Let
~$\gamma=\gamma(\mathcal{U})=\sup\{|\varphi(x)-\varphi(y)|:
x,y\in U, U\in \mathcal{U}\}. $
Then for every $\alpha\in \mathbb{R}$ and $N>0$
$$M(Z,\alpha,\varphi,\mathcal{U},N)\leq M(Z,\alpha-\gamma,\varphi,\mathcal{V},N).$$
We have that
$$P_{Z}(\varphi, \mathcal{U})-\gamma\leq P_{Z}(\varphi, \mathcal{V}). $$
Since $X$ has admissible cover of arbitrarily small diameter, then
$$P_{Z}(\varphi, \mathcal{U})-\gamma\leq \mathop{\underline{\lim}}\limits_{|\mathcal{V}|\rightarrow 0}
P_{Z}(\varphi, \mathcal{V}). $$
If ~$|\mathcal{U}|\rightarrow 0$, then ~$\gamma(\mathcal{U})\rightarrow 0$ and hence
$$\mathop{\overline{\lim}}\limits_{|\mathcal{U}|\rightarrow 0}P_{Z}(\varphi, \mathcal{U})
\leq \mathop{\underline{\lim}}\limits_{|\mathcal{V}|\rightarrow 0}P_{Z}(\varphi, \mathcal{V}).$$
This implies the existence of the first limit.
The others can be proved in a similar fashion.
\end{proof}

We call the quantities $P_{Z}(\varphi)$, $\overline{CP}_{Z}(\varphi,f)$, $\underline{CP}_{Z}(\varphi)$
respectively, the topological pressure and lower and upper capacity topological pressure of the function $\varphi$ on the set $Z$(with respect to $f$). We write $P_{Z,f}(\varphi)$, $\overline{CP}_{Z,f}(\varphi,f)$,$\underline{CP}_{Z,f}(\varphi)$ respectively to emphasize $f$ if we need to.

By the basic properties of the C-P structure\cite{Pesin}, we get the following basic properties.

\begin{theorem}\label{c}
Let~$(X,d)$ be a metric space and $d$ an admissible metric, $f:X\rightarrow X$ a proper map. Let
$\varphi\in C(X,\mathbb{R})$ be bounded. Then

(1)~ $P_{\emptyset}(\varphi)\leq0$.

(2)~ $P_{Z_{1} }(\varphi)\leq P_{Z_{2}}(\varphi)$ if ~$Z_1\subset Z_2\subset X$.

(3)~ $P_{Z}(\varphi)=\sup_{i\geq1}P_{Z_{i}}(\varphi)$, where~$Z=\bigcup_{i\geq1}Z_{i},~Z_{i}\subset X,
i=1, 2, \cdots$.

(4)~ If ~$f$ is a homeomorphism then ~$P_{Z}(\varphi)=P_{f(Z)}(\varphi)$, where ~$Z$ is any subset of~$X$.

\end{theorem}

\begin{theorem}
Let~$(X,d)$ be a metric space and $d$ an admissible metric, $f:X\rightarrow X$ a proper map. Let $\varphi\in C(X,\mathbb{R})$ be bounded. Then

(1)~ $\underline{CP}_{\emptyset}(\varphi)\leq0,~\overline{CP}_{\emptyset}(\varphi)\leq0$;

(2)~ ~$\underline{CP}_{Z_{1}}(\varphi)\leq \underline{CP}_{Z_{2}}(\varphi),~ \overline{CP}_{Z_{1}}(\varphi)\leq \overline{CP}_{Z_{2}}(\varphi)$ if ~$Z_1\subset Z_2\subset X$.

(3)~ $\underline{CP}_{Z}(\varphi)\geq \sup_{i\geq1}\underline{CP}_{Z_{i}}(\varphi)$ and
$\overline{CP}_{Z}(\varphi)\geq \sup_{i\geq1}\overline{CP}_{Z_{i}}(\varphi)$, where
~$Z=\bigcup_{i\geq1}{Z_{i}}$,~$Z_{i}\subset X, i=1,2,\cdots $.
\end{theorem}

\begin{theorem}
Let~$(X,d)$ be a metric space and $d$ an admissible metric, $f:X\rightarrow X$ a proper map. Let $\varphi\in C(X,\mathbb{R})$ be bounded and ~$\mathcal{U}$ an admissible cover of ~$X$. Then for any $Z\subset X,$ we have that
$$\underline{CP}_{Z}(\varphi,\mathcal{U})=\mathop{\underline{\lim}}\limits_{N\rightarrow+\infty}
\frac{1}{N}{\log\Lambda(Z,\varphi,\mathcal{U},N)},
$$
$$\overline{CP}_{Z}(\varphi,\mathcal{U})=\mathop{\overline{\lim}}\limits_{N\rightarrow+\infty}
\frac{1}{N}{\log\Lambda(Z,\varphi,\mathcal{U},N)},
$$
where
$$\Lambda(Z,\varphi,\mathcal{U},N)=\inf_{\mathcal{G}}\left \{\sum_{\mathbf{U}\in \mathcal{G}}\exp\left
(\sup_{x\in X(\mathbf{U})}(S_{N}\varphi)(x)\right)\right\},
$$
the infimum is taken over all finite or countable collections of strings $\mathcal{G}\subset \mathcal{S}(\mathcal{U})$ such that ~$m(\mathbf{U})= N$ for all ~$\mathbf{U}\in \mathcal{G}$ and $\mathcal{G}$ covers ~$Z$.
\end{theorem}

We can use the analogous methods as that of \cite{Pesin} to prove the following theorems, so we omit the proof.

\begin{theorem}\label{MMMM}
Let~$(X,d)$ be a metric space and $d$ an admissible metric, $f:X\rightarrow X$ a proper map. Let $\varphi,\psi \in C(X,\mathbb{R})$ be bounded. Then for any ~$Z\subset X$, we have that
$$\mid P_{Z}(\varphi)-P_{Z}(\psi)\mid\leq\parallel \varphi-\psi \parallel ,$$
$$\mid \underline{CP}_{Z}(\varphi)-\underline{CP}_{Z}(\psi)\mid\leq\parallel \varphi-\psi \parallel ,$$
$$\mid \overline{CP}_{Z}(\varphi)-\overline{CP}_{Z}(\psi)\mid\leq\parallel \varphi-\psi \parallel ,$$
where $\parallel \varphi \parallel=\sup_{x\in X} \mid \varphi(x)\mid.$
\end{theorem}

\begin{theorem}~~
Let~$(X,d)$ be a metric space and $d$ an admissible metric, $f:X\rightarrow X$ a proper map. Let $\varphi\in C(X,\mathbb{R})$ be bounded.

(1)~ For any ~$Z\subset X,$ if ~$f^{-1}(Z)=Z$, then for any admissible open cover ~$\mathcal{U}$ of ~$X,$
~$\underline{CP}_{Z}(\varphi,\mathcal{U})=\overline{CP}_{Z}(\varphi,\mathcal{U})$. Moreover,
~$\underline{CP}_{Z}(\varphi)=\overline{CP}_{Z}(\varphi).$

(2)~ As the same conditions as that of (1), if ~$Z$ is compact set, then
~$P_{Z}(\varphi, \mathcal{U})=\underline{CP}_{Z}(\varphi,\mathcal{U})=\overline{CP}_{Z}(\varphi,\mathcal{U}).$ Moreover,
~$P_{Z}(\varphi)=\underline{CP}_{Z}(\varphi)=\overline{CP}_{Z}(\varphi).$
\end{theorem}

\begin{remark}
(1)If $X$ is a compact metric space, then these notions of topological pressure coincide with those notions of topological pressure introduced by Pesin and Pitskel\cite{Pesin,PP}.

(2)If $\varphi=0,$ then these notions of topological pressure coincide with those notions of topological entropy introduced by Ma and Cai\cite{MC}. If $Z=X, \varphi=0,$ then $\underline{CP}_{X}(0)=\overline{CP}_{X}(0)$
coincide the topological entropy introduced by Patr\~{a}o\cite{Patr}.

(3)It is easy to see that
$$P_{Z}(\varphi)\leq\underline{CP}_{Z}(\varphi)\leq\overline{CP}_{Z}(\varphi),~~ \forall Z\subset X.$$
\end{remark}

\section{Some further properties of the topological pressure and lower and upper capacity topological pressure}

In this section, we give further properties of these topological pressures.

 If $X$ is a locally compact separable metric space, we can associate $X$ with its one-point compactification, which is denoted by $\widetilde{X}$. We have that $\widetilde{X}$ is defined as the disjoint union of $X$ with $\{\infty\}$, where $\infty$ is some point not in $X$ called the point at infinity. The topology in $\widetilde{X}$ consist of the former open sets in $X$ and the sets $A\cup\{\infty\}$, where the complement of $A$ in $X$ is compact. Let $f:X\rightarrow X$ be a proper map. Defining $\widetilde{f}:\widetilde{X}\rightarrow\widetilde{X}$ by
$$\widetilde{f}(\widetilde{x})=\left\{\begin{array}{ll}
f(\widetilde{x}),& \widetilde{x}\neq\infty\\
\infty,& \widetilde{x}=\infty,
\end{array}
\right.
$$
we have that $\widetilde{f}$ is also a proper map, called the extension of $f$ to $\widetilde{X}$. We note that the separability of $X$ is equivalent to the metrizability of $\widetilde{X}$.

\begin{lemma}(\cite{Patr})\label{3}~~~
Let ~$X$ be a locally compact separable metric space and $d$ the metric given by the restriction to $X$ of some metric ~$\widetilde{d}$ on ~$\widetilde{X}$, the one-point compactification of ~$X$. Then $d$ is an admissible metric. \end{lemma}

\begin{theorem}\label{4}~~
Let~$X$ be a locally compact separable metric space and $d$ the metric given by the restriction to $X$ of some metric ~$\widetilde{d}$ on ~$\widetilde{X}$, the one-point compactification of ~$X$. Let $f:X\rightarrow X$ be a proper map.
~$\varphi\in C(X,\mathbb{R})$ can be continuously extended to ~$\widetilde{X}$ denoted by ~$\widetilde{\varphi}$, then for any~$Z\subset X$, we have the that
$$P_{Z,f}(\varphi)=P_{Z,\widetilde{f}}^{PP}(\widetilde{\varphi}),$$
$$\underline{CP}_{Z,f}(\varphi)=\underline{CP}_{Z,\widetilde{f}}^{PP}(\widetilde{\varphi}),$$
$$\overline{CP}_{Z,f}(\varphi)=\overline{CP}_{Z,\widetilde{f}}^{PP}(\widetilde{\varphi}),$$
where $P_{Z,\widetilde{f}}^{PP}(\widetilde{\varphi}),\underline{CP}_{Z,\widetilde{f}}^{PP}(\widetilde{\varphi})$ and
$\overline{CP}_{Z,\widetilde{f}}^{PP}(\widetilde{\varphi})$ denote the Pesin-Pitskel topological pressure, lower and upper capacity topological pressure \cite{Pesin,PP}, respectively.
\end{theorem}

\begin{proof}

Let $d$ be the metric given by the restriction to $X$ of some metric ~$\widetilde{d}$ on ~$\widetilde{X}$. By Lemma \ref{3}, we have that $d$ is an admissible metric. Let~$\mathcal{U}=\{U_0,U_1,\cdots,U_{k-1}\}$ be an admissible cover of ~$X.$ Let
$\widetilde{U}_i=\{y:\exists x\in U_i,~\widetilde{d}(x,y)<
| U_i|\},~0\leq i\leq k-1,$
then ~$\widetilde{\mathcal{U}}=\{\widetilde{U}_0,\widetilde{U}_1,\cdots,\widetilde{U}_{k-1}\}$
is an open cover of ~$\widetilde{X}$ and ~$|\mathcal{U}|\rightarrow 0$ implies ~$|\widetilde{\mathcal{U}}|\rightarrow0.$
Let ~$\mathcal{G}\subset\bigcup_{n\geq N}\mathcal{S}_n(\mathcal{U})$
covers a set ~$Z\subset X.$ For any ~$\mathbf{U}=(U_{i_0}U_{i_1}\cdots U_{i_{n-1}})\in\mathcal{G},$
define~$\widetilde{\mathbf{U}}=(\widetilde{U}_{i_0}\widetilde{U}_{i_1}\cdots
\widetilde{U}_{i_{n-1}})\in\mathcal{S}_{n}(\widetilde{\mathcal{U}})$,
and denote the collection of all these strings by ~$\widetilde{\mathcal{G}}$, then ~$\widetilde{\mathcal{G}}$ covers
~$Z\subset\widetilde{X}.$ Moreover, we have that
$$\widetilde{M}(Z,\alpha,\widetilde{\varphi},\widetilde{\mathcal{U}},N)
\leq M(Z,\alpha-\gamma,\varphi,\mathcal{U},N),$$
where~$\gamma=\gamma(\widetilde{\mathcal{U}})=\sup\{|\widetilde{\varphi}(x)-
\widetilde{\varphi}(y)|:x,y\in\widetilde{U},\widetilde{U}\in\widetilde{\mathcal{U}}\}.$
Let ~$N\rightarrow\infty$, then
$$\widetilde{m}(Z,\alpha,\widetilde{\varphi},\widetilde{\mathcal{U}})
\leq m(Z,\alpha-\gamma,\varphi,\mathcal{U}).$$
This implies that
$$P_{Z,\widetilde{f}}^{PP}(\widetilde{\varphi},\widetilde{\mathcal{U}})-\gamma
\leq P_{Z,f}(\varphi,\mathcal{U}).$$
Let ~$|\mathcal{U}|\rightarrow0,$ then ~$|\widetilde{\mathcal{U}}|\rightarrow0,$ ~$\gamma\rightarrow0,$ so we have
$$P_{Z,f}(\varphi)\geq P_{Z,\widetilde{f}}^{PP}(\widetilde{\varphi}).$$

We are going to show that~$P_{Z,f}(\varphi)\leq P_{Z,\widetilde{f}}^{PP}(\widetilde{\varphi}).$
If ~$\widetilde{\mathcal{U}}_{\frac{\varepsilon}{2}}=\{\widetilde{B}(\widetilde{x}_{0},\frac{\varepsilon}{2})$, $\cdots,\widetilde{B}(\widetilde{x}_{k-1},\frac{\varepsilon}{2})\}$
is a cover of ~$\widetilde{X},$ for every ~$\varepsilon\in(a,b),$ where ~$0<a<b.$ By the density of $X$ in  ~$\widetilde{X},$ it follows that there exist
~$\{x_{0},\cdots, x_{k-1}\}\subset X$, such that
~$\widetilde{d}(x_{i},\widetilde{x}_{i})<\frac{\varepsilon}{2},~0\leq i \leq k-1.$
If ~$x\in X$, we have that ~$\widetilde{d}(x,\widetilde{x}_{i})<\frac{\varepsilon}{2},$ for some
~$\widetilde{x}_{i}\in \{\widetilde{x}_{0},\cdots, \widetilde{x}_{k-1}\}.$ Hence it follows that
~$d(x,x_{i})\leq \widetilde{d}(x,\widetilde{x}_{i})+\widetilde{d}(x_{i},\widetilde{x}_{i})<\varepsilon,$
showing that ~$\{B(x_{0},\varepsilon),\cdots,B(x_{k-1},\varepsilon)\}$
is a cover of ~$X.$ Since ~$d$ is an admissible metric, there exists ~$\delta\in(a,b)$ such that
~$\mathcal{U}_{\delta}:=\{B(x_{0},$ $\delta),\cdots,B(x_{k-1},\delta)\}$
is an admissible cover of ~$X.$ For ~$a< \varepsilon< \delta < b$, we have that
$$M(Z,\alpha,\varphi,\mathcal{U}_{\delta},N)\leq \widetilde{M}(Z,\alpha-\gamma,
\widetilde{\varphi},\widetilde{\mathcal{U}}_{\frac{\varepsilon}{2}},N),$$
where ~$\gamma=\gamma(\mathcal{U}_\delta)=\sup\{\mid\varphi(x)-\varphi(y)\mid:
x,y\in B(x_{i},\delta),B(x_{i},\delta)\in\mathcal{U}_\delta\}.$
Let ~$N\rightarrow\infty$, we have
$$m(Z,\alpha,\varphi,\mathcal{U}_\delta )\leq \widetilde{m}(Z,\alpha-\gamma,
\widetilde{\varphi},\widetilde{\mathcal{U}}_{\frac{\varepsilon}{2}} ).$$
Moreover, $P_{Z,f}(\varphi,\mathcal{U}_\delta)-\gamma \leq
P_{Z,\widetilde{f}}^{PP}(\widetilde{\varphi},\widetilde{\mathcal{U}}_{\frac{\varepsilon}{2}}).$
Let $b\rightarrow 0,$ then $\varepsilon\rightarrow 0,$ $\delta\rightarrow 0$ and $\gamma\rightarrow 0.$
Then
$$P_{Z,f}(\varphi)\leq P_{Z,\widetilde{f}}^{PP}(\widetilde{\varphi}).$$
Hence ~$P_{Z,f}(\varphi)= P_{Z,\widetilde{f}}^{PP}(\widetilde{\varphi}).$ The others can be proved in a similar fashion.
\end{proof}

\begin{remark}
It is easy to see that this theorem is an extension of Theorem 4.3 in \cite{MC}.
\end{remark}

\begin{theorem}\label{5}~~
Let~$X$ be a locally compact separable metric space and $d$ the metric given by the restriction to $X$ of some metric ~$\widetilde{d}$ on ~$\widetilde{X}$, the one-point compactification of ~$X$. Let $f:X\rightarrow X$ be a proper map.
~$\varphi\in C(X,\mathbb{R})$ can be continuously extended to ~$\widetilde{X}$ denoted by ~$\widetilde{\varphi}$, then
$$P_{X,f}(\varphi)=P_{\widetilde{X},\widetilde{f}}^{PP}(\widetilde{\varphi}),$$
$$\underline{CP}_{X,f}(\varphi)=\underline{CP}_{\widetilde{X},\widetilde{f}}^{PP}(\widetilde{\varphi}),$$
$$\overline{CP}_{X,f}(\varphi)=\overline{CP}_{\widetilde{X},\widetilde{f}}^{PP}(\widetilde{\varphi}).$$

\end{theorem}

\begin{proof}
From Theorem \ref{4} and a property of Pesin-Pitskel pressure, we have
$$ P_{X,f}(\varphi)=P_{X,\widetilde{f}}^{PP}(\widetilde{\varphi})
\leq P_{\widetilde{X},\widetilde{f}}^{PP}(\widetilde{\varphi}). $$
On the other hand,
Let ~$\mathcal{U}=\{U_0,U_1,\cdots,U_{k-1}\}$ be an admissible cover of ~$X.$ Define
$\widetilde{U}_i=\{y:\exists x\in U_i,~\widetilde{d}(x,y)<
|U_i|\},~0\leq i\leq k-1,$
then ~$\widetilde{\mathcal{U}}=\{\widetilde{U}_0,\widetilde{U}_1,\cdots,\widetilde{U}_{k-1}\}$
is a cover of ~$\widetilde{X}$ and if ~$|\mathcal{U}|\rightarrow 0$ then ~$|\widetilde{\mathcal{U}}|\rightarrow0.$
Suppose ~$\mathcal{G}\subset\bigcup_{n\geq N}\mathcal{S}_n(\mathcal{U})$ cover ~$X$. For any ~$\mathbf{U}=(U_{i_0}U_{i_1}\cdots U_{i_{m-1}})\in\mathcal{G},$ define ~$\widetilde{\mathbf{U}}:=(\widetilde{U}_{i_0}\widetilde{U}_{i_1}\cdots
\widetilde{U}_{i_{m-1}})$. Denote the all this strings by ~$\widetilde{\mathcal{G}}$, i.e., $\widetilde{\mathcal{G}}=\{\widetilde{\mathbf{U}}:\mathbf{U}\in \mathcal{G} \}.$  Let ~$\varepsilon=\min_{0\leq i\leq k-1}|U_i|$,
then there exists ~$x\in X$, such that ~$\widetilde{d}(x,\infty)<\varepsilon,
~\widetilde{d}(\widetilde{f}(x),\widetilde{f}(\infty))<\varepsilon, \cdots,
~\widetilde{d}(\widetilde{f}^{m-1}(x),\widetilde{f}^{m-1}(\infty))<\varepsilon .$
Since ~$\mathcal{G}$ covers ~$X$, then there exists
~$\mathbf{U}_{0}=(U_{i_0}U_{i_1}\cdots U_{i_{m-1}})\in\mathcal{G}$, such that ~$x\in X(\mathbf{U}_{0}),$
Let ~$\widetilde{\mathbf{U}}_{0}=(\widetilde{U}_{i_0}\widetilde{U}_{i_1}\cdots
\widetilde{U}_{i_{m-1}})\in \widetilde{\mathcal{G}}$,
then ~$\infty\in \widetilde{X}(\widetilde{\mathbf{U}}_{0}).$ So ~$\widetilde{\mathcal{G}}$ covers
~$\widetilde{X}$. Hence
$$\widetilde{M}(\widetilde{X},\alpha,\widetilde{\varphi},\widetilde{\mathcal{U}},N)
\leq M(X,\alpha-\gamma,\varphi,\mathcal{U},N),$$
where ~$\gamma=\gamma(\widetilde{\mathcal{U}})=\sup\{|\widetilde{\varphi}(x)-
\widetilde{\varphi}(y)|:x,y\in\widetilde{U},\widetilde{U}\in\widetilde{\mathcal{U}}\}.$
Let ~$N\rightarrow\infty$, we have
$$\widetilde{m}(\widetilde{X},\alpha,\widetilde{\varphi},\widetilde{\mathcal{U}})
\leq m(X,\alpha-\gamma,\varphi,\mathcal{U}).$$
Moreover
$$P_{\widetilde{X},\widetilde{f}}^{PP}(\widetilde{\varphi},\widetilde{\mathcal{U}})-\gamma
\leq P_{X,f}(\varphi,\mathcal{U}).$$
If ~$|\mathcal{U}|\rightarrow0,$ then ~$|\widetilde{\mathcal{U}}|\rightarrow0$ and ~$\gamma\rightarrow0$ and hence
$$P_{\widetilde{X},\widetilde{f}}^{PP}(\widetilde{\varphi})\leq P_{X,f}(\varphi).$$
So
$$P_{\widetilde{X},\widetilde{f}}^{PP}(\widetilde{\varphi})= P_{X,f}(\varphi).$$
The others can be proved in a similar fashion.
\end{proof}

\begin{remark}
We can see that this theorem extends a result of Patr\~{a}o\cite{Patr}.
\end{remark}

\begin{corollary}~~
Let~$X$ be a locally compact separable metric space and $d$ the metric given by the restriction to $X$ of some metric ~$\widetilde{d}$ on ~$\widetilde{X}$, the one-point compactification of ~$X$. Let $f:X\rightarrow X$ be a proper map.
~$\varphi \in C(X,\mathbb{R})$ can be continuously extended to ~$\widetilde{X}$ denoted by ~$\widetilde{\varphi}$, then for any ~$n\in \mathbb{N},$ we have
$$P_{X,f^k}(S_k \varphi)=kP_{X,f}(\varphi).$$

\end{corollary}

\begin{proof}
For any ~$n\in \mathbb{N}$, by the Theorem \ref{5} we have
$$P_{X,f^k}(S_k \varphi)=P_{\widetilde{X},\widetilde{f}^k}^{PP}(S_k \widetilde{\varphi}),$$
$$P_{X,f}(\varphi)= P_{\widetilde{X},\widetilde{f}}^{PP}(\widetilde{\varphi}). $$
From \cite{Walters} we have that
$$P_{\widetilde{X},\widetilde{f}^k}^{PP}(S_k \widetilde{\varphi})
=kP_{\widetilde{X},\widetilde{f}}^{PP}(\widetilde{\varphi}). $$
So
$$P_{X,f^k}(S_k \varphi)=kP_{X,f}(\varphi).$$
\end{proof}

\begin{remark}
Under the conditions of Theorem \ref{5}, one can also extend the other properties of the classical topological pressure\cite{Walters}.
\end{remark}

\section{Some variational principles.}

In this section we present some variational principles involving the topological pressure.

\begin{theorem}(partial variational principle)\label{a}
Let~$X$ be a locally compact separable metric space and $d$ the metric given by the restriction to $X$ of some metric ~$\widetilde{d}$ on ~$\widetilde{X}$, the one-point compactification of ~$X$. Let $f:X\rightarrow X$ be a proper map.
~$\varphi\in C(X,\mathbb{R})$ can be continuously extended to ~$\widetilde{X}$ denoted by ~$\widetilde{\varphi}$, ~$M(X,f)\neq \emptyset$,
then
$$P_{X}(\varphi)\geq \sup\left\{h_\mu(f)+\int\varphi d\mu:\mu\in M(X,f)\right\}, $$
where $M(X,f)$ denotes the set of all $f-$invariant probability measures on $X$ and $h_{\mu}(f)$ is the measure-theoretic entropy of $f$ with respect to $\mu\in M(X,f).$
\end{theorem}

\begin{proof}
If $\mu\in M(X,f)$, define $\widetilde{\mu}(\widetilde{A})=\mu(\widetilde{A}\cap X).$ It is immediate that $\widetilde{\mu}\in M(\widetilde{X},\widetilde{f}),$ since $X$ and $\{\infty\}$ are $\widetilde{f}-$invariant sets. It is also immediate that $\widetilde{\mu}(\{\infty\})=0,$
$\int\widetilde{\varphi}~d\widetilde{\mu}=\int\varphi~d\mu$ and $h_{\mu}(f)=h_{\tilde{\mu}}(\widetilde{f}).$
By Theorem \ref{5} and the classical variational principle \cite{Walters}, we have that
\begin{eqnarray*}
&&P_{X,f}(\varphi)= P_{\widetilde{X},\widetilde{f}}^{PP}(\widetilde{\varphi}) \\
&=&\sup\left\{h_{\widetilde{\mu}}(\widetilde{f})+\int_{\widetilde{X}}\widetilde{\varphi}~d\widetilde{\mu}:
\widetilde{\mu} \in M(\widetilde{X},\widetilde{f})\right\} \\
&\geq&\sup\left\{h_\mu(f)+\int_{X}\varphi~d\mu:\mu\in M(X,f)\right\}.
\end{eqnarray*}
\end{proof}

\begin{remark}
Under the conditions of Theorem \ref{a}. If $\varphi=0,$ Patr\~{a}o\cite{Patr} proved that
\begin{equation}\label{b}
\sup\left\{h_\mu(f):\mu\in M(X,f)\right\}=\sup\left\{h_{\widetilde{\mu}}(\widetilde{f}):
\widetilde{\mu} \in M(\widetilde{X},\widetilde{f})\right\}
\end{equation}
and $h(f)=h(\widetilde{f}),$ where $h(f)$ denotes the Patr\~{a}o topological entropy\cite{Patr} and $h(\widetilde{f})$ denotes the classical topological entropy\cite{AKM}. By the classical variational principle\cite{Walters}, we have that $$h(\widetilde{f})= \sup\left\{h_{\widetilde{\mu}}(\widetilde{f}):
\widetilde{\mu} \in M(\widetilde{X},\widetilde{f})\right\}.$$
Combining (\ref{b}) we have that $$P_{X,f}(0)=P_{\widetilde{X},\widetilde{f}}^{PP}(0)=h(\widetilde{f})=h(f)=\sup\left\{h_\mu(f):\mu\in M(X,f)\right\},$$
i.e., if $\varphi=0,$ then the equality holds in Theorem\ref{a}.
We can construct an example which says that the inequality holds in Theorem\ref{a}.
Let $X=\mathbb{R},$ $f:X\rightarrow X$ be a homeomorphism, defining $\varphi:X\rightarrow\mathbb{R}$ by
$$\varphi(x)=\left\{\begin{array}{ll}
\text{arccot}(x),& x< 0\\
\text{arccot}(-x),& x\geq 0,
\end{array}
\right.
$$
then $\widetilde{\varphi}:\widetilde{X}\rightarrow \mathbb{R},$
$$\widetilde{\varphi}(x)=\left\{\begin{array}{ll}
\varphi(x),& x\in X\\
\pi,& x= \infty.
\end{array}
\right.
$$
The graph of $\varphi$ is in Figure 1. $\widetilde{X}$ and the unit circle are homeomorphic.
Since $h(f)=\sup\left\{h_\mu(f):\mu\in M(X,f)\right\}=\sup\left\{h_{\widetilde{\mu}}(\widetilde{f}):
\widetilde{\mu} \in M(\widetilde{X},\widetilde{f})\right\}=h(\widetilde{f})=0,$ then $h_\mu(f)=h_{\widetilde{\mu}}(\widetilde{f})=0,\forall \mu\in M(X,f),\forall \widetilde{\mu}\in M(\widetilde{X},\widetilde{f}).$
Let $\widetilde{\mu}_{1}\in M(\widetilde{X},\widetilde{f})$ such that $\widetilde{\mu}_{1}(\{\infty\})=1.$
Then $\int_{\widetilde{X}}\widetilde{\varphi}~d\widetilde{\mu}_{1}=\pi.$ Moreover,
$$\sup\left\{h_{\widetilde{\mu}}(\widetilde{f})+\int_{\widetilde{X}}\widetilde{\varphi}~d\widetilde{\mu}:
\widetilde{\mu} \in M(\widetilde{X},\widetilde{f})\right\}\geq \int_{\widetilde{X}}\widetilde{\varphi}~d\widetilde{\mu}_{1}=\pi.$$
On the other hand,
$$\sup\left\{h_\mu(f)+\int_{X}\varphi~d\mu:\mu\in M(X,f)\right\}=\sup\left\{\int_{X}\varphi~d\mu:\mu\in M(X,f)\right\}< \pi.$$

$$
\includegraphics[height=8cm]{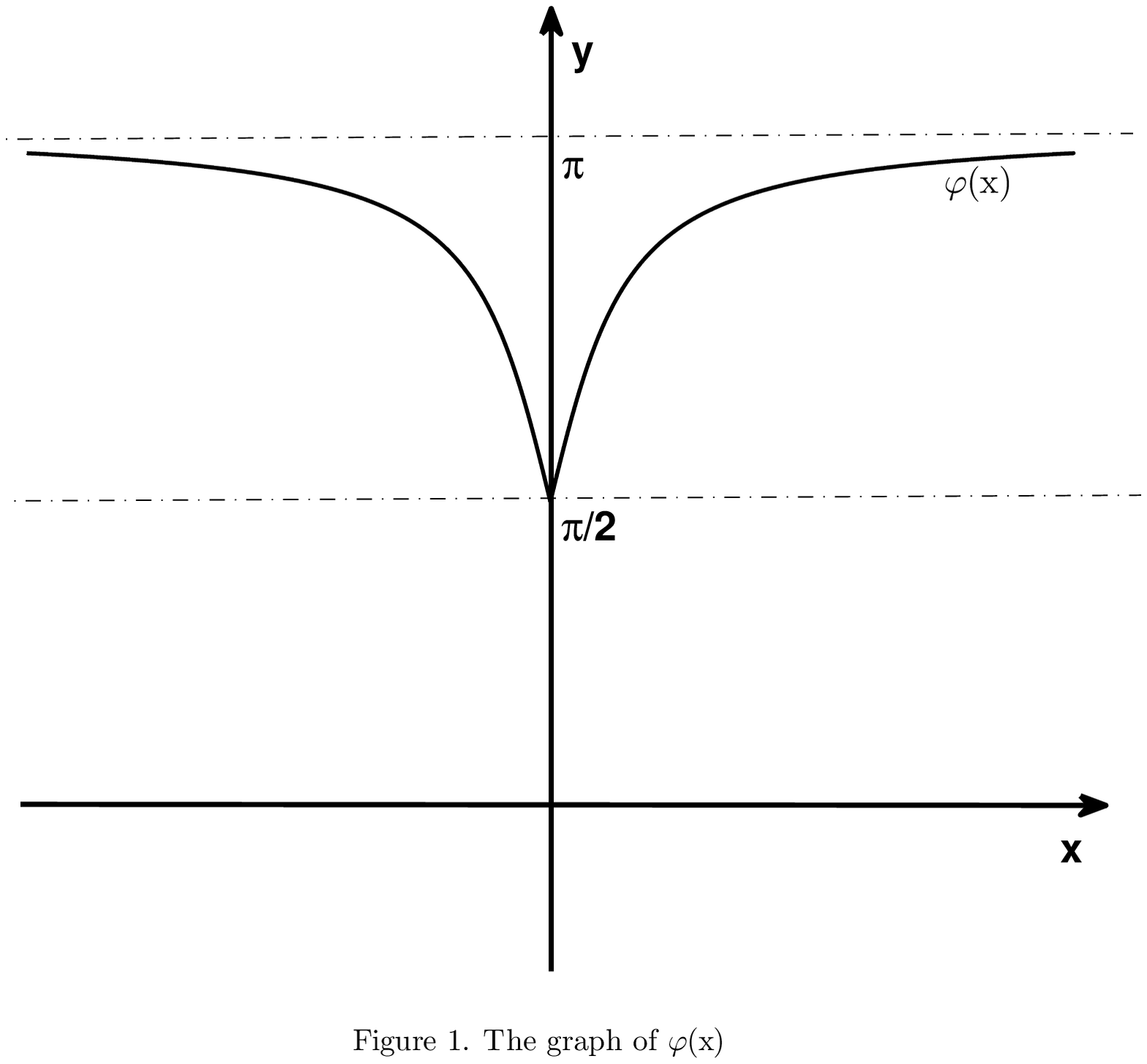}
$$

\end{remark}

\begin{theorem}(Inverse variational principle)
Let~$X$ be a locally compact separable metric space and $d$ the metric given by the restriction to $X$ of some metric ~$\widetilde{d}$ on ~$\widetilde{X}$, the one-point compactification of ~$X$. Let $f:X\rightarrow X$ be a proper map.
~$\varphi\in C(X,\mathbb{R})$ can be continuously extended to ~$\widetilde{X}$ denoted by ~$\widetilde{\varphi}$, ~$M(X,f)\neq \emptyset$,
then for any ~$\mu\in M(X,f)$, we have that
$$h_\mu(f)+\int\varphi d\mu=\inf\{P_{Z,f}(\varphi):Z\subset X,\mu(Z)=1 \}.$$
\end{theorem}

\begin{proof}
If ~$\mu\in M(X,f)$, defining ~$\widetilde{\mu}(\widetilde{A})=\mu(\widetilde{A}\cap X)$,
then ~$\widetilde{\mu}\in M(\widetilde{X},\widetilde{f})$ and ~$h_\mu(f)=h_{\widetilde{\mu}}(\widetilde{f}).$
If $Z\subset X$ and $\mu(Z)=1,$ then $\widetilde{\mu}(Z)=1$. If $\widetilde{Z}\subset \widetilde{X},\infty\in \widetilde{Z}$ and $\widetilde{\mu}(\widetilde{Z})=1$, then $\mu(\widetilde{Z}\cap X)=1$ and $\widetilde{\mu}(\{\infty\})=0. $
Pesin\cite{Pesin} gave that $$h_{\widetilde{\mu}}(\widetilde{f})+\int_{\widetilde{X}}\widetilde{\varphi}d\widetilde{\mu}
=\inf\{P_{\widetilde{Z},\widetilde{f}}^{PP}(\widetilde{\varphi}):\widetilde{Z}\subset\widetilde{X},
\widetilde{\mu}(\widetilde{Z})=1  \}.$$
Applying Theorem \ref{4} and the Pesin's result, we have that
\begin{eqnarray*}
h_\mu(f)+\int\varphi d\mu
&=&h_{\widetilde{\mu}}(\widetilde{f})+\int_{\widetilde{X}}\widetilde{\varphi}d\widetilde{\mu}\\
&=&\inf\{P_{\widetilde{Z},\widetilde{f}}^{PP}(\widetilde{\varphi}):\widetilde{Z}\subset\widetilde{X},
\widetilde{\mu}(\widetilde{Z})=1  \}\\
&=&\inf\{P_{Z,\widetilde{f}}^{PP}(\widetilde{\varphi}):Z\subset X, \mu(Z)=1 \}\\
&=&\inf\{P_{Z,f}(\varphi):Z\subset X,\mu(Z)=1 \}.
\end{eqnarray*}
\end{proof}

\begin{remark}
It is easy to see that this theorem extends a result of Pesin\cite{Pesin}.
\end{remark}

\section{Multifractal analysis of local entropies for expansive homeomorphism with specification.}

In this section, as some applications of the topological pressure of a proper map for a locally compact separable metric space, we give some results of multifractal analysis of local entropies for expansive homeomorphisms with specification. These results extend the results of Takens and Verbitski \cite{TV}.

Let $(X,d)$ be a metric space and $f:X\rightarrow X$ a proper map, we say that $f$ with the specification property if for any $\varepsilon>0$ there exists an integer $m=m(\varepsilon)$ such that for arbitrary finite intervals $I_{j}=[a_{j},b_{j}]\subset \mathbb{N}, j=1,\cdots,k,$ such that
$$\text{dist}(I_{i},I_{j})\geq m(\varepsilon),i\neq j,$$
and any $x_{1},\cdots,x_{k}$ in $X$ there exists a point $x\in X$ such that
$$d(f^{p+a_{j}}(x),f^{p}(x_{j}))<\varepsilon$$
for all $p=0,\cdots,b_{j}-a_{j}$ and every $j=1,\cdots,k.$

\begin{lemma}(\cite{MC})\label{9}
Let $X$ be a locally compact separable metric space, $f: X\rightarrow X$ be a proper map with the specification property respect to the metric that is the restriction of some metric on $\widetilde{X}$. Then $\widetilde{f}$ satisfies the specification property too, where $\widetilde{f}$ is the extension of $f$.
\end{lemma}

Let $X$ be a metric space and $f:X\rightarrow X$ a homeomorphism. An admissible cover $\alpha$ of $X$ is called a generator for $f$ if for every bisequence $\{A_{n}\}_{-\infty}^{\infty}$ of members of $\alpha$ the set $\bigcap_{n=-\infty}^{\infty}f^{-n}\bar{A}_{n}$ contains at most one point of $X.$ If this condition is replaced by $\bigcap_{n=-\infty}^{\infty}f^{-n}A_{n}$ contains at most one point of $X$ then $\alpha$ is called a weak generator.

\begin{lemma}\label{6}
If $f: X\rightarrow X$ is a homeomorphism of a locally compact separable metric space, then $f$ has a generator iff $f$ has a weak generator.
\end{lemma}

\begin{proof}
A generator is clearly a weak generator. Now suppose $\beta$ is a weak generator for $f,$ $\beta=\{B_1,\cdots,B_s\}.$
Suppose $d$ is the restriction of some metric $\widetilde{d}$ on $\widetilde{X},$ the one-point compactification of $X.$ By the Lemma \ref{2}, $\beta$ has a Lebesgue number. denoted by $\delta.$ By the Lemma \ref{3}, there exists an admissible cover of $X,$ denoted by $\alpha$ such that the diameter of $\alpha$ is less than $\delta.$ So if $\{A_{n}\}_{-\infty}^{\infty}$ is a bisequence in
$\alpha$ then for any $n$ there exists $j_{n}$ with $\bar{A}_{i_{n}}\subseteq B_{j_{n}}.$ Hence
$$\bigcap_{n=-\infty}^{\infty}f^{-n}\bar{A}_{i_{n}}\subseteq \bigcap_{n=-\infty}^{\infty}f^{-n}A_{j_{n}},$$
which is either empty or a single point. So $\alpha$ is a generator.
\end{proof}

Let $(X,d)$ be a metric space and $f:X\rightarrow X$ a homeomorphism,
we say that $f$ is expansive if there exists ~$\delta>0$ with the property that if ~$x\neq y$ then there exists ~$n\in\mathbb{Z}$ with ~$d(f^{n}(x),~f^{n}(y))>\delta$. We call ~$\delta$ an expansive constant for ~$f.$

\begin{lemma}\label{7}
Let $f: X\rightarrow X$ be a homeomorphism of a locally compact separable metric space $(X,d)$ and $d$ the restriction of some metric $\widetilde{d}$ on $\widetilde{X},$ the one-point compactification of $X.$ Then $f$ is expansive iff $f$ has a generator iff $f$ has a weak generator.
\end{lemma}

\begin{proof}
By Lemma \ref{6} it suffices to show $f$ is expansive iff $f$ has a generator. Let $\delta$ be an expansive constant for $f$ and $\alpha$ an admissible cover of radius less than  $\delta/2.$ Suppose $x,y\in \bigcap_{n=-\infty}^{\infty}f^{-n}\bar{A}_{n}$ where $A_{n}\in \alpha.$ Then $d(f^{n}(x),~f^{n}(y))\leq\delta$, $\forall n\in \mathbb{Z}$. So, by assumption $x=y.$ Therefore $\alpha$ is a generator.
On the other hand, suppose $\alpha$ is a generator. Let $\delta$ be a Lebesgue number for $\alpha.$ If $d(f^{n}(x),~f^{n}(y))\leq\delta$, $\forall n\in \mathbb{Z}$. Then for $\forall n\in \mathbb{Z}$,
there exists $A_{n}\in \alpha$ with $f^{n}(x)$, $f^{n}(y)\in A_{n}$ and so,
$x,y\in \bigcap_{n=-\infty}^{\infty}f^{-n}A_{n}$.
Since this intersection contains at most one point we have $x=y.$ Hence $f$ is expansive.
\end{proof}

\begin{remark}
Lemma \ref{6} and Lemma \ref{7} are the extensions of the results of the compact systems\cite{Walters}.
\end{remark}

\begin{lemma}\label{8}
Let $f: X\rightarrow X$ be an expansive homeomorphism of a locally compact separable metric space $(X,d),$ where $d$ is the restriction of some metric $\widetilde{d}$ on $\widetilde{X},$ the one-point compactification of $X.$ Then $\widetilde{f}$ is an expansive homeomorphism, where $\widetilde{f}$ is the extension of $f.$
\end{lemma}

\begin{proof}
By the Lemma \ref{7}, $f$ has a weak generator. Now suppose $\alpha$ is the weak generator, $\alpha=\{A_{1},\cdots,A_{k}\}.$ By Lemma \ref{1},
$\alpha$ has Lebesgue number.
Let $\delta$ be a Lebesgue number for $\alpha.$ Denote the open ball in $\widetilde{X}$ of radius $\delta/2$ centered in $\infty$ by $\widetilde{B}(\infty,\delta/2).$ Let $\beta=\{A_{1},\cdots,A_{k},\widetilde{B}(\infty,\delta/2)\},$ then $\beta$ is an open cover of $\widetilde{X}.$
For every bisequence $\{B_{n}\}_{-\infty}^{\infty}$ of members of $\beta,$ if $\widetilde{B}(\infty,\delta/2)\not\in\{B_{n}\}_{-\infty}^{\infty},$ then $\bigcap_{n=-\infty}^{\infty}f^{-n}B_{n}$
contains at most one point. If $\widetilde{B}(\infty,\delta/2)\in\{B_{n}\}_{-\infty}^{\infty},$ suppose $\widetilde{B}(\infty,\delta/2)=B_{k},$ then there exists $A_{j}\in \alpha,$ such that $B_{k}-\{\infty \}\subseteq A_{j}.$ Then
\begin{eqnarray*}
\bigcap_{n=-\infty}^{\infty}f^{-n}B_{n}&=& f^{-k}B_{k}\cap\bigcap_{n=-\infty,n\neq k}^{\infty}f^{-n}B_{n}\\
&\subseteq &(\{\infty\}\cup f^{-k}A_{j})\cap\bigcap_{n=-\infty,n\neq k}^{\infty}f^{-n}B_{n}\\
&=& f^{-k}A_{j}\cap\bigcap_{n=-\infty,n\neq k}^{\infty}f^{-n}B_{n}
\end{eqnarray*}
also contains at most one point.
Then $\beta$ is a weak generator of $\widetilde{f}$ and $\widetilde{f}$ is expansive.

\end{proof}

Following Krin and Katok \cite{BK}, we introduce the notions of local entropy of a proper map.

\begin{definition}
Let ~$(X,d)$ be a metric space and $f:X\rightarrow X$ a proper map, we introduce the lower and upper local entropies at ~$x\in X$ as follows
$$\underline{h}_{\mu}(f,x):=\lim_{\varepsilon\rightarrow 0}\liminf_{n\rightarrow\infty}
-\frac{1}{n}\log\mu(B_{n}(x,\varepsilon)),$$
$$\overline{h}_{\mu}(f,x):=\lim_{\varepsilon\rightarrow 0}\limsup_{n\rightarrow\infty}
-\frac{1}{n}\log\mu(B_{n}(x,\varepsilon)),$$
where ~$\mu$ is a Borel probability measure on ~$X$ and $B_{n}(x,\varepsilon)$ is the Bowen ball of~$x.$
We say that the local entropy exists at $x$ if
$$\underline{h}_{\mu}(f,x)=\overline{h}_{\mu}(f,x).$$
In this case the common value will be denoted by ~$h_{\mu}(f,x).$

\end{definition}

Similar to the compact systems, we introduce the following notions.

\begin{definition}
Let ~$f:X\rightarrow X$ be a proper map of metric space~$(X,d)$ and
~$\varphi\in C(X,\mathbb{R})$.
A member ~$\mu\in M(X,f)$ is called an equilibrium state for ~$\varphi$
if $$ P_{X}(\varphi)=h_{\mu}(f)+\int\varphi d\mu .$$
\end{definition}

\begin{definition}
Let ~$(X,d)$ be a metric space and $f:X\rightarrow X$ a proper map, we say that ~$\varphi \in
\mathcal{V}_f (X)$ if it is continuous and there exist ~$\varepsilon>0$ and ~$K>0$ such that for all
~$n\in \mathbb{N},$
$$d(f^k(x),f^k(y))<\varepsilon, k=0,\cdots,n-1
\Rightarrow |S_n(\varphi)(x)-S_n(\varphi)(y)|< K. $$
\end{definition}

Following \cite{BPS} and \cite{TV}, we introduce a multifractal spectrum for (local) entropies. For
every ~$\alpha$ consider a level set of local entropy
$$K_\alpha=\{x\in X: h_\mu(f,x)=\alpha \} , $$
and the corresponding multifractal decomposition on level sets
$$ X=\bigcup_{\alpha} K_\alpha \bigcup\{ x\in X: h_\mu(f,x)\text{does not exist}\}.$$
We can use the topological entropy of proper map introduced in \cite{MC}, to measure the size of sets$\{K_\alpha \}.$
Namely, define a multifractal spectrum for local entropies as follows:
$$\mathcal{E}_E (\alpha)=h_{K_\alpha}(f).$$

For compact systems, we have the following result.

\begin{theorem}(\cite{Bowen1,Ruelle,KH})\label{*}
Let ~$X$ be a compact metric space and ~$f:X\rightarrow X$ an expansive homeomorphism with specification and $\varphi\in \mathcal{V}_{f} (X)$, then there exists a unique measure ~$\mu_{\varphi}$ such that
   $$P_{X}(\varphi)=h_{\mu_{\varphi}}(f)+\int\varphi d\mu_{\varphi} ,$$
where $P_{X}(\varphi)$ is the classical topological pressure \cite{Walters} and coincide with $P_{X}^{PP}(\varphi).$ Moreover, ~$\mu_{\varphi}$ is ergodic, positive on open sets and mixing.
\end{theorem}

\begin{theorem}(\cite{TV})\label{***}
Let ~$f$ be an expansive homeomorphism with the specification property of a compact metric space ~$(X,d).$ Let
~$ \varphi\in\mathcal{V}_f (X)$ and $\mu=\mu_\varphi$ be the corresponding equilibrium state. Then

(1)~For ~$\mu-a.e.x\in X$ the local entropy at ~$x$ exists and
$$h_\mu(f,x)=h_\mu(f)=P_{X}(\varphi)-\int\varphi d\mu. $$

(2)~ For any ~$q\in\mathbb{R}$ define the function
$$ T(q)=P_{X}(q\varphi)-qP_{X}(\varphi). $$
Then ~$T(q)$ is a convex ~$\mathcal{C}^1$ function of ~$q.$ Moreover, ~$T(0)=h(f),~T(1)=0,$
for every ~$q\in \mathbb{R}$ one has ~$T^{'}(q)=\int\varphi d\mu_q-P_{X}(\varphi)\leq0,$
where ~$\mu_q$ is the equilibrium state for ~$\varphi_q=q\varphi-P_{X}(q\varphi),$
$P_{X}(\varphi)$ is the classical topological pressure\cite{Walters}, $h(f)$ is the classical topological entropy\cite{AKM}.

(3)~ Put ~$\alpha(q)=-T^{'}(q).$ Then
$$\mathcal{E}_E (\alpha(q)):=h_{K_{\alpha(q)}}(f)=T(q)+q\alpha(q),$$
where $h_{Z}(f)$ denotes the Bowen topological entropy \cite{Bowen} of any subset $Z\subseteq X.$
Define
$$\underline{\alpha}=\inf_q \alpha(q)=\lim_{q\rightarrow +\infty}\alpha(q),$$
$$\overline{\alpha}=\sup_q \alpha(q)=\lim_{q\rightarrow -\infty}\alpha(q).$$
Then ~$K_\alpha=\emptyset$ if ~$\alpha \not\in [\underline{\alpha},~\overline{\alpha}]$.
It means that the domain of the multifractal spectrum for local entropies $\alpha\rightarrow \mathcal{E}_{E}(\alpha)$ is the range of the function $q\rightarrow -T^{'}(q).$

(4)~If the equilibrium state ~$\mu$ for the potential ~$\varphi$ is not a measure of maximal entropy, then the relation between ~$\mathcal{E}_E$ and ~$T(q)$ can be written in the following variational form:
$$ \mathcal{E}_E(\alpha)=\inf_{q\in\mathbb{R}}(T(q)+q\alpha),~~~
\alpha\in(\underline{\alpha},\overline{\alpha}),$$
$$T(q)=\sup_{\alpha\in(\underline{\alpha},\overline{\alpha})}
(\mathcal{E}_E(\alpha)-q\alpha),~~~ q\in\mathbb{R}.$$
This implies that ~$\mathcal{E}_E$ is strictly concave and continuously differentiable on ~$(\underline{\alpha},\overline{\alpha})$ with the derivative given by ~$\mathcal{E}_{E}^{'}(\alpha)=q,$
 where ~$q\in\mathbb{R}$ is such that ~$\alpha=-T^{'}(q).$

(5)~ For every ~$q\in\mathbb{R},~q\neq1$, the following limit exists:
$$ h_\mu(f,q)=\lim_{\varepsilon\rightarrow0}\lim_{n\rightarrow\infty}
-\frac{1}{n(q-1)}\log\int\mu(B_n(x,\varepsilon))^{q-1}d\mu.$$
For ~$q\neq1$ one has
$$ h_\mu(f,q)=-\frac{T(q)}{q-1}.$$
The family of correlation entropies ~$h_\mu(f,q)$ depends continuously on ~$q$ and
$$h_\mu(f,0)=h(f),$$
$$h_\mu(f,1):=\lim_{q\rightarrow1}h_\mu(f,q)=h_\mu(f).$$

\end{theorem}

\begin{lemma}\label{**}
Let $f: X\rightarrow X$ be an expansive homeomorphism with specification of a locally compact separable metric space $(X,d),$ where $d$ is the restriction of some metric $\widetilde{d}$ on $\widetilde{X},$ the one-point compactification of $X.$ If $ \varphi\in C(X,\mathbb{R})$ can be continuously extended to $\widetilde{X}$ denoted by  ~$\widetilde{\varphi},$ $\widetilde{\varphi}\in\mathcal{V}_{\widetilde{f}}(\widetilde{X}).$ Then there exists an equilibrium state for $\varphi.$
\end{lemma}

\begin{proof}
By Lemma \ref{8} and Lemma \ref{9}, we have that $\widetilde{f}:\widetilde{X}\rightarrow \widetilde{X}$ be an expansive homeomorphism with specification. Since $\widetilde{\varphi}\in\mathcal{V}_{\widetilde{f}}(\widetilde{X}),$ by Theorem \ref{*}, there exists an unique equilibrium state $\widetilde{\mu},$ such that
$$P_{\widetilde{X},\widetilde{f}}(\widetilde{\varphi})=h_{\widetilde{\mu}}(\widetilde{f})+
\int_{\widetilde{X}}\widetilde{\varphi}d\widetilde{\mu}.$$
Moreover, $\widetilde{\mu}$ is ergodic and positive on open sets. So $$\widetilde{\mu}(\{\infty\})=0.$$
Define $\mu=\widetilde{\mu}|_{\mathcal{B}(X)},$ then $\mu\in M(X,f)$ and $h_{\widetilde{\mu}}(\widetilde{f})=h_{\mu}(f),$
$\int_{\widetilde{X}}\widetilde{\varphi}d\widetilde{\mu}=\int_{X}\varphi d \mu.$
By Theorem \ref{5}, we have $P_{\widetilde{X},\widetilde{f}}(\widetilde{\varphi})=P_{X,f}(\varphi).$
Then $P_{X,f}(\varphi)=h_{\mu}(f)+\int_{X}\varphi d \mu.$
\end{proof}

Based on the Lemma \ref{**}, we have the following result, which is the extension of Theorem \ref{***}.

\begin{theorem}
Let $f: X\rightarrow X$ be an expansive homeomorphism with specification of a locally compact separable metric space $(X,d),$ where $d$ is the restriction of some metric $\widetilde{d}$ on $\widetilde{X},$ the one-point compactification of $X.$ If $ \varphi\in C(X,\mathbb{R})$ can be continuously extended to $\widetilde{X}$ denoted by  ~$\widetilde{\varphi},$ $\widetilde{\varphi}\in\mathcal{V}_{\widetilde{f}}(\widetilde{X})$ and ~$\mu$ is the corresponding equilibrium state for ~$\varphi,$ then

(1)~ For ~$\mu-a.e.x\in X$ the local entropy at ~$x$ exists and
$$h_\mu(f,x)=h_\mu(f)=P_{X,f}(\varphi)-\int_{X}\varphi d\mu. $$

(2)~ For any ~$q\in\mathbb{R}$ define the function
$$ T(q)=P_{X,f}(q\varphi)-qP_{X,f}(\varphi).$$
Then ~$T(q)$ is a convex ~$\mathcal{C}^1$ function of ~$q.$ Moreover ~$T(0)=h(f),~T(1)=0;$
for every ~$q\in \mathbb{R}$ one has ~$T^{'}(q)=\int_{X}\varphi~d\mu_q-P_{X,f}(\varphi)\leq0,$
where ~$\mu_q$ is the equilibrium state for ~$\varphi_q=q\varphi-P_{X,f}(q\varphi),$ $h(f)$ is the Patr\~{a}o topological entropy \cite{Patr}.

(3)~Put ~$\alpha(q)=-T^{'}(q)$. Then
$$\mathcal{E}_E (\alpha(q)):=h_{K_{\alpha(q)}}(f)=T(q)+q\alpha(q).$$

Define
$$\underline{\alpha}=\inf_q \alpha(q)=\lim_{q\rightarrow +\infty}\alpha(q),$$
$$\overline{\alpha}=\sup_q \alpha(q)=\lim_{q\rightarrow -\infty}\alpha(q).$$
Then ~$K_\alpha=\emptyset$ if ~$\alpha\not\in[\underline{\alpha},\overline{\alpha}].$
It means that the domain of the multifractal spectrum for local entropies $\alpha\rightarrow \mathcal{E}_{E}(\alpha)$ is the range of the function $q\rightarrow -T^{'}(q),$ where $h_{Z}(f)$ denote the Ma-Cai topological entropy\cite{MC}.

(4)~If the equilibrium state ~$\mu$ for the potential ~$\varphi$ is not a measure of maximal entropy, then the relation between ~$\mathcal{E}_E$ and ~$T(q)$ can be written in the following variational form:
$$ \mathcal{E}_E(\alpha)=\inf_{q\in\mathbb{R}}(T(q)+q\alpha),~~~
\alpha\in(\underline{\alpha},\overline{\alpha}),$$
$$T(q)=\sup_{\alpha\in(\underline{\alpha},\overline{\alpha})}
(\mathcal{E}_E(\alpha)-q\alpha),~~~ q\in\mathbb{R}.$$
This implies that ~$\mathcal{E}_E$ is strictly concave and continuously differentiable on ~$(\underline{\alpha},\overline{\alpha})$ with the derivative given by ~$\mathcal{E}_{E}^{'}(\alpha)=q,$
 where ~$q\in\mathbb{R}$ is such that ~$\alpha=-T^{'}(q).$

(5)~ For every ~$q\in\mathbb{R},~q\neq1$, the following limit exists:
$$ h_\mu(f,q)=\lim_{\varepsilon\rightarrow0}\lim_{n\rightarrow\infty}
-\frac{1}{n(q-1)}\log\int\mu(B_n(x,\varepsilon))^{q-1}d\mu.$$
For ~$q\neq1$ one has
$$ h_\mu(f,q)=-\frac{T(q)}{q-1}.$$
The family of correlation entropies ~$h_\mu(f,q)$ depends continuously on ~$q$ and
$$h_\mu(f,0)=h(f),$$
$$h_\mu(f,1):=\lim_{q\rightarrow1}h_\mu(f,q)=h_\mu(f).$$
\end{theorem}

\begin{proof}
(1)~By Lemma \ref{8} and Lemma \ref{9} , we have that ~$\widetilde{f}:\widetilde{X}\rightarrow\widetilde{X}$ be an expansive homeomorphism with specification. By Lemma \ref{*}
, $\widetilde{\varphi}$ has an equilibrium state ~$\widetilde{\mu}$,
Applying Theorem \ref{***}, we have that for ~$\widetilde{\mu}-a.e.x\in \widetilde{X}$ the local entropy exists and
$$h_{\widetilde{\mu}}(\widetilde{f},x)=h_{\widetilde{\mu}}(\widetilde{f})
=P_{\widetilde{X},\widetilde{f}}(\widetilde{\varphi})- \int_{\widetilde{X}}\widetilde{\varphi}~d\widetilde{\mu}.$$
By Lemma \ref{**}, $\varphi$ has an equilibrium state $\mu=\widetilde{\mu}|_{\mathcal{B}(X)},$
$h_{\widetilde{\mu}}(\widetilde{f})=h_{\mu}(f),$ $\int_{\widetilde{X}}\widetilde{\varphi}~d\widetilde{\mu}=\int_{X}\varphi d\mu.$ Moreover, for $\mu-a.e.x\in X,$
$h_{\widetilde{\mu}}(\widetilde{f},x)=h_{\mu}(f,x).$
Combining Theorem \ref{5}, we have that
for ~$\mu-a.e.x\in X$ the local entropy exists and
$$h_\mu(f,x)=h_\mu(f)=P_{X,f}(\varphi)-\int_{X}\varphi d\mu. $$

(2)~ By Theorem \ref{5}, we have that ~$P_{X,f}(\varphi)=P_{\widetilde{X},\widetilde{f}}(\widetilde{\varphi}),$ then
$$T(q)=P_{X,f}(q\varphi)-qP_{X,f}(\varphi)=P_{\widetilde{X},\widetilde{f}}(q\widetilde{\varphi})
-qP_{\widetilde{X},\widetilde{f}}(\widetilde{\varphi}):=\widetilde{T}(q), \forall q\in \mathbb{R}.$$
Applying Theorem \ref{***}, we have that ~$T(q)$ is a convex ~$\mathcal{C}^1$ function of ~$q,$
~$T(0)=\widetilde{T}(0)
=h(\widetilde{f}).$ By a result of Patr\~{a}o\cite{Patr}, we have that $h(\widetilde{f})=h(f),$ then $T(0)=h(f).$ It is easy to see that $~T(1)=0.$ For every $q\in \mathbb{R}$ one has
\begin{eqnarray*}
T^{'}(q)&=&\widetilde{T}^{'}(q)\\
&=&\int_{\widetilde{X}}\widetilde{\varphi}d\widetilde{\mu}_{q}-P_{\widetilde{X},\widetilde{f}}(\widetilde{\varphi})\\
&=&\int_{X}\varphi d\mu_{q}-P_{X,f}(\varphi)\leq 0,
\end{eqnarray*}
where $\widetilde{\mu}_{q}$ is the equilibrium state for $\widetilde{\varphi}_q=q\widetilde{\varphi}-P_{\widetilde{X},\widetilde{f}}(q\widetilde{\varphi})$ and ~$\mu_q=\widetilde{\mu}_{q}|_{\mathcal{B}(X)}$ is the equilibrium state for ~$\varphi_q=q\varphi-P_{X}(q\varphi).$

(3) Let ~$\alpha(q)=-T^{'}(q)$, then from (2) we have that $\alpha(q)=-\widetilde{T}^{'}(q).$
Then
$$\widetilde{K}_{\alpha(q)}=\{x\in \widetilde{X}:
h_{\widetilde{\mu}}(\widetilde{f},x)=\alpha(q)\}
=\{x\in X:h_{\mu}(f,x)=\alpha(q)\}=K_{\alpha(q)} ,$$
or
$$\widetilde{K}_{\alpha(q)}=\{x\in \widetilde{X}:
h_{\widetilde{\mu}}(\widetilde{f},x)=\alpha(q)\}
=\{x\in X:h_{\mu}(f,x)=\alpha(q)\}\cup\{\infty\}=K_{\alpha(q)}\cup\{\infty\} .$$
Combining the fact $h_{\{\infty\}}(\widetilde{f})=0,$ Theorem \ref{4} and Theorem \ref{c}, we have that
$$\mathcal{E}_E (\alpha(q)):=h_{K_{\alpha(q)}}(f)=h_{K_{\alpha(q)}}(\widetilde{f})
=h_{\widetilde{K}_{\alpha(q)}}(\widetilde{f})
=\widetilde{T}(q)+q\alpha(q)=T(q)+q\alpha(q).
$$
The others hold from Theorem\ref{***}.

(4) For any ~$\alpha\in(\underline{\alpha},\overline{\alpha}) $, we have that
$$\widetilde{K}_{\alpha}=\{x\in \widetilde{X}:
h_{\widetilde{\mu}}(\widetilde{f},x)=\alpha\}
=\{x\in X:h_{\mu}(f,x)=\alpha\}=K_{\alpha} ,$$
or
$$\widetilde{K}_{\alpha}=\{x\in \widetilde{X}:
h_{\widetilde{\mu}}(\widetilde{f},x)=\alpha\}
=\{x\in X:h_{\mu}(f,x)=\alpha\}\cup\{\infty\}=K_{\alpha}\cup\{\infty\} .$$
If the equilibrium state ~$\mu$ for the potential ~$\varphi$ is not a measure of maximal entropy, then combining Theorem \ref{c}, Theorem \ref{4} and Theorem \ref{***}, we have that
$$\mathcal{E}_E (\alpha):=h_{K_{\alpha}}(f)=h_{K_{\alpha}}(\widetilde{f})=h_{\widetilde{K}_{\alpha}}(\widetilde{f})
=\inf_{q\in\mathbb{R}}(\widetilde{T}(q)+q\alpha)=\inf_{q\in\mathbb{R}}(T(q)+q\alpha),$$
$$T(q)=\widetilde{T}(q)=\sup_{\alpha\in(\underline{\alpha},\overline{\alpha})}
(\mathcal{E}_E(\alpha)-q\alpha).$$
The others hold from Theorem\ref{***}.

(5) Since $\mu=\widetilde{\mu}|_{\mathcal{B}(X)}$ and $\widetilde{\mu}(\{\infty\})=0,$ one has
$$ \lim_{\varepsilon\rightarrow0}\lim_{n\rightarrow\infty}
-\frac{1}{n(q-1)}\log\int_{X}\mu(B_n(x,\varepsilon))^{q-1}d\mu $$
$$=\lim_{\varepsilon\rightarrow0}\lim_{n\rightarrow\infty}
-\frac{1}{n(q-1)}\log\int_{\widetilde{X}}\widetilde{\mu}(\widetilde{B}_n(x,\varepsilon))^{q-1}d\widetilde{\mu}.$$
By Theorem \ref{***}, the second limit exists, so the first limit exists too, which denoted by
~$h_\mu(f,q)$. Then ~$h_\mu(f,q)=h_{\widetilde{\mu}}(\widetilde{f},q). $
The other three equations hold from Theorem \ref{***}.

\end{proof}

{\bf Acknowledgement.} The work was partially supported by the Fundamental Research Fund for the Central Universities (2013ZZ0085).

\newpage

\bibliographystyle{amsplain}

\end{document}